\documentclass{mystyle}

\title[The Hermitian Axiom on 2D TQFTs]{The Hermitian Axiom on Two-Dimensional Topological Quantum Field Theories}

\author[H. Zhu]{Honglin Zhu}

\begin{document}

\begin{abstract}
We examine Atiyah's Hermitian axiom for two-dimensional complex topological quantum field theories. Building on the correspondence between 2D TQFTs and Frobenius algebras, we find the algebraic objects corresponding to Hermitian and unitary TQFTs respectively and prove structure theorems about them. We then clarify a few older results on unitary TQFTs using our structure theorems. 
\end{abstract}

\maketitle
\section{Introduction}
A topological quantum field theory in dimension $d$, by Atiyah's axiomatic formulation in (\cite{atiyah1989}), is a rule $Z$ which assigns a vector space $Z(\Sigma)$ to each closed oriented $(d-1)$-dimensional manifold $\Sigma$ (throughout the paper, all manifolds are understood to be compact, smooth, and orientable), and to each oriented smooth $d$-dimensional manifold $M$ with boundary $\Sigma$ an element $Z(M) \in Z(\Sigma)$ (see \Cref{sec_axioms}). 

Another model of a $d$-dimensional TQFT is to view it as a linear representation of $dCob$, the category of $d$-dimensional cobordisms. A cobordism from an oriented $(d-1)$-dimensional manifold $\Sigma_1$ to another oriented $(d-1)$-dimensional manifold $\Sigma_2$ is an oriented $d$-dimensional manifold $M$ whose boundary is $\Sigma_1^* \cup \Sigma_2$, that is, the orientation of $M$ matches with that of the out-boundary $\Sigma_2$ and is the opposite of the orientation of the in-boundary $\Sigma_1$. We will write $\Sigma^*$ for the $(d-1)$-dimensional manifold $\Sigma$ with orientation reversed. For a cobordism $M$ from $\Sigma_1$ to $\Sigma_2$, we will write $M^*$ for the manifold $M$ with its orientation reversed. By the involutory axiom and the multiplicative axiom, the element $Z(M)$ can be regarded as a homomorphism from $Z(\Sigma_1)$ to $Z(\Sigma_2)$. Thus, a cobordism $M$ from $\Sigma_1$ to $\Sigma_2$ can be viewed as a linear map from $Z(\Sigma_1)$ to $Z(\Sigma_2)$. For more discussions on the cobordism category $dCob$ and its connection to TQFTs, see (\cite{kock2003}).

In dimension two, there is a fascinating result that establishes an equivalence of categories between two-dimensional TQFTs over a field $F$ and Frobenius algebras over $F$. Thus, for each 2D TQFT there is a corresponding Frobenius algebra, and conversely each Frobenius algebra gives rise to a 2D TQFT. This gives us a third way to understand TQFTs, and in particular, it allows us to use algebraic tools to study something geometric in nature. For more information, Kock offers a detailed examination of 2D TQFTs and Frobenius algebras in (\cite{kock2003}).

In the case of TQFTs over the complex numbers, when the vector spaces $Z(\Sigma)$ are equipped with nondegenerate Hermitian forms, Atiyah proposes an extra axiom (\cref{herm}) which relates a cobordism $M$ with $M^*$, namely, that the linear map $Z(M)$ is the adjoint of $Z(M^*)$. We call a TQFT satisfying this axiom a Hermitian TQFT and a Hermitian TQFT with positive-definite Hermitian forms a unitary TQFT (\Cref{herm_tqft_def}). 

The main focus of this paper is to relate the Hermitian and unitary conditions in each of the three models for TQFTs. Starting from Atiyah's axioms, we will derive the extra algebraic structures that the Hermitian and unitary conditions impose on the corresponding Frobenius algebras, as well as clarify various arguments in the literature.

\subsection*{Acknowledgements}
This paper is the result of a direct funding project in the Undergraduate Research Opportunities Program at MIT. The author would like to thank Professor Haynes Miller for his mentorship in this project. He would like to thank Professor Stephen Sawin and the referee for providing helpful comments. He would also like to thank the J-WEL Grant in Higher Education Innovation, ``Educational Innovation in Palestine,'' which provided the initial impetus for this research. Finally, the author would like to thank the MIT UROP office for providing this research opportunity.

\section{Atiyah's Axioms}\label{sec_axioms}
The axiomatic formulation of topological quantum field theories where first proposed by Michael Atiyah in (\cite{atiyah1989}). In this paper, Atiyah defines a TQFT in dimension $d$ over a field $F$ as a rule $Z$ which:
\begin{enumerate}[label=(\Alph*)]
    \item associates to each compact oriented smooth $(d-1)$-dimensional manifold $\Sigma$ a finite dimensional $F$-vector space $Z(\Sigma)$, and
    \item associates to each compact oriented smooth $d$-dimensional manifold $M$ with boundary $\partial M$ an element $Z(M) \in Z(\partial M)$.
\end{enumerate}
The rule $Z$ is subject to the following TQFT axioms:
\begin{enumerate}
    \item $Z$ is \textit{functorial} with respect to orientation preserving diffeomorphisms of $\Sigma$ and $M$, 
    \item $Z$ is \textit{involutory}, that is, there is a natural isomorphism $Z(\Sigma^*) \cong Z(\Sigma)^*$ where $\Sigma^*$ is $\Sigma$ with orientation reversed and $Z(\Sigma)^*$ denotes the dual space,
    \item $Z$ is \textit{multiplicative}, 
    \item for $\varnothing$ the empty $(d-1)$-dimensional manifold, $Z(\varnothing) = F$; for $\varnothing$ the empty $d$-dimensional manifold, $Z(\varnothing) = 1$; for a cylinder $\Sigma \times I$, $Z(\Sigma \times I) \in Z(\Sigma)^* \otimes Z(\Sigma) \cong \operatorname{End}(Z(\Sigma))$ is the identity map.
\end{enumerate}
Atiyah offers the following elaboration on his axioms in (\cite{atiyah1989}).

The functorial axiom means that an orientation preserving diffeomorphism $f: \Sigma \to \Sigma'$ induces an isomophism $Z(f): Z(\Sigma) \to Z(\Sigma')$ and for some other diffeomorphism $g: \Sigma' \to \Sigma''$, $Z(g \circ f) = Z(g) \circ Z(f)$. If $f$ extends to an orientation preserving diffeomorphism $M \to M'$ with $\partial M = \Sigma$, $\partial M' = \Sigma'$, then $Z(M) \mapsto Z(M')$ under $Z(f)$.

The multiplicative axiom asserts that for disjoint unions, there is a natural isomorphism 
\[
    Z(\Sigma_1 \cup \Sigma_2) \cong Z(\Sigma_1) \otimes Z(\Sigma_2).
\]
Furthermore, if $M_1$, $M_2$ have respective boundaries $\Sigma_1 \cup \Sigma_3$, $\Sigma_2 \cup \Sigma_3^*$, and we glue them together at the $\Sigma_3$ boundary to get $M = M_1 \cup_{\Sigma_3} M_2$, then 
\[
    Z(M) = \langle Z(M_1), Z(M_2) \rangle,
\]
where $\langle \cdot, \cdot \rangle$ denotes the natural pairing $Z(\Sigma_1) \otimes Z(\Sigma_3) \otimes Z(\Sigma_3)^* \otimes Z(\Sigma_2) \to Z(\Sigma_1) \otimes Z(\Sigma_2)$. In particular, when $\Sigma_3 = \varnothing$ so that $M$ is the disjoint union of $M_1$ and $M_2$, $Z(M) = Z(M_1) \otimes Z(M_2)$. 

Viewing the multiplicative axiom from the cobordism category perspective, we get an equivalent and helpful reformulation. If a manifold $M$ has boundary $\Sigma_1^* \cup \Sigma_2$, then we can view it as a cobordism from $\Sigma_1$ to $\Sigma_2$, with either boundary component possibly empty. Then the first part of the axiom states that $M$ is a linear transformation
\[
    Z(M) \in Z(\Sigma_1)^* \otimes Z(\Sigma_2) \cong \operatorname{Hom}(Z(\Sigma_1), Z(\Sigma_2)).
\]
The second part then asserts that when we glue two cobordisms together we get the composition of their respective linear transformations. 

\section{The Hermitian Axiom}
In (\cite{atiyah1989}), after formulating the TQFT axioms, Atiyah notes that they have yet to give a relation between the values of the TQFT on $M$ and $M^*$. Thus, he proposes an additional axiom (\cref{herm}) that provides such a relation, in the case the ground field is the field $\mathbb{C}$ of complex numbers. To state the Hermitian axiom, we assume that for each $(d-1)$-dimensional manifold $\Sigma$, the vector space $Z(\Sigma)$ comes equipped with a nondegenerate Hermitian structure. That is, there is a nondegenerate Hermitian form $h_{\Sigma}: Z(\Sigma) \times Z(\Sigma) \to \mathbb{C}$, linear in the first argument and conjugate-linear in the second. The Hermitian form $h_{\Sigma}$ on $Z(\Sigma)$ induces a conjugate linear isomorphism $H: Z(\Sigma) \to Z(\Sigma)^*$ through $H(v) = h_{\Sigma}(\cdot, v)$. Then for a manifold $M$ with boundary $\partial M$, the Hermitian axiom states that $Z(M^*)$ is identified with $H(Z(M))$ under the isomorphism from $Z(\partial M^*)$ to $Z(\partial M)^*$ provided by the involutory axiom:
\begin{equation} \label{herm}
    Z(M^*) = h_{\partial M} (\cdot, Z(M)).
\end{equation}

These Hermitian structures on the vector spaces associated to $(d-1)$-dimensional manifolds are required to satisfy the following natural properties:
\begin{enumerate}\label{herm_properties}
    \item If $Z(\Sigma)$ has Hermitian form $h$, then $Z(\Sigma)^*$ has the Hermitian form 
    \[
        h^*(h(\cdot, v), h(\cdot, w)) = h(w,v)
    \]
    for all $h(\cdot, v), h(\cdot, w) \in Z(\Sigma)^*$;
    \item If $Z(\Sigma_1)$, $Z(\Sigma_2)$ have Hermitian forms $h_1$, $h_2$, respectively, then $Z(\Sigma_1) \otimes Z(\Sigma_2)$ has the Hermitian form 
    \[
        h_{1\otimes 2}(v_1 \otimes v_2, w_1 \otimes w_2) = h_1(v_1, w_1) h_2(v_2, w_2)
    \]
    for pure tensors $v_1 \otimes v_2, w_1 \otimes w_2 \in Z(\Sigma_1) \otimes Z(\Sigma_2)$.
\end{enumerate}

The cobordism model gives a convenient reformulation of the Hermitian axiom.
\begin{theorem}\label{herm_adj}
    Suppose $M$ is a cobordism from $\Sigma_1$ to $\Sigma_2$. Let $V = Z(\Sigma_1)$, $W = Z(\Sigma_2)$ so that we can view $Z(M)$ and $Z(M^*)$ as linear transformations $Z(M): V \to W$ and $Z(M^*): W \to V$. If $h_{V}$, $h_{W}$ are the Hermitian forms on $V$ and $W$, respectively, then $Z(M^*)$ is the adjoint of $Z(M)$. That is,
    \[
        h_{W}(Z(M) (v), w) = h_{V}(v, Z(M^*) (w))
    \]
    for all $v \in V$, $w \in W$.
\end{theorem}
\begin{proof}
    Without loss of generality, suppose $Z(M) = h_V(\cdot, x) \otimes y$ is a pure tensor in $V^* \otimes W$. Then by the Hermitian axiom (\cref{herm}), $Z(M^*) = x \otimes h_W(\cdot, y)$. Then for any $v \in V$, $w \in W$,
    \begin{align*}
         h_{W}(Z(M) (v), w) 
         &= h_{W}(h_V(v, x) \cdot y, w)\\
         &= h_{V}(v, x) \cdot h_{W}(y, w)\\
         &= h_{V}(v, x \cdot h_{W}(w, y))\\
         &= h_{V}(v, Z(M^*) (w)).
    \end{align*}
\end{proof}

We use the following terminology throughout the rest of the paper. 
\begin{definition}\label{herm_tqft_def}
    Let $Z$ be a TQFT over the complex numbers. We say $Z$ is \textit{Hermitian} if $Z$ satisfies the Hermitian axiom (\cref{herm}). 
    
    If the Hermitian forms $h_{\Sigma}$ are all positive-definite, following the convention in (\cite{durhuus1994}), we say $Z$ is \textit{unitary}. 
\end{definition}
These definitions make sense in all finite dimensions, but in the remainder of the paper, we will focus on two-dimensional Hermitian and unitary TQFTs. 

\section{The Algebraic Structure of Hermitian TQFTs} \label{sec_hermitian}
In dimension $d = 2$, the Hermitian axiom (\cref{herm}) has an algebraic interpretation through the conditions it imposes on the commutative Frobenius algebra associated to a TQFT. 

\begin{definition}[Kock (\cite{kock2003})]\label{frob_def}
    A \textit{Frobenius algebra} $A$ is a finite-dimensional unital commutative $F$-algebra equipped with a linear functional $\epsilon: A \to F$ whose nullspace contains no nontrivial ideals. The functional $\epsilon \in A^*$ is called a \textit{Frobenius form}. 
\end{definition}

The following proposition relates Frobenius algebras with TQFTs. 
\begin{proposition}[Kock (\cite{kock2003})]\label{frob_prop}
    A Frobenius algebra $A$ can be seen as its underlying vector space $V$ equipped with a multiplication $\mu: V \otimes V \to V$ with unit $\eta: F \to V$ and a Frobenius form $\epsilon \in V^*$. Then there exists a unique comultiplication $\delta: V \to V \otimes V$ with counit $\epsilon$ such that the Frobenius relation holds:
    \begin{equation}\label{eq_frob}
        (\operatorname{id} \otimes \mu)(\delta \otimes \operatorname{id}) = \delta \mu = (\mu \otimes \operatorname{id})(\operatorname{id} \otimes \delta).
    \end{equation}
\end{proposition}
In this paper, we will fix the notation $A = (V, \mu, \eta, \delta, \epsilon)$ for a Frobenius algebra $A$ with underlying vector space $V$, multiplication $\mu$, unit $\eta$, comultiplication $\delta$, and counit $\epsilon$. Following the convention in (\cite{kock2003}), we also write $\beta = \epsilon \mu: V \otimes V \to F$ for the pairing and $\gamma = \delta \eta: F \to V \otimes V$ for the copairing.

Since the pairing is a nondegenerate symmetric bilinear form on the vector space $V = Z(S^1)$, it is natural to suspect that the presence of both a Hermitian form and a bilinear form, in addition to some compatibility condition provided by the Hermitian axiom, induces a natural real structure on $V$. This is indeed the case, as we will illustrate in this section. 

\begin{definition}\label{star_frob_def}
    A complex Frobenius algebra $A$ is a \textit{$^*$-Frobenius algebra} if $A$ is also a $^*$-algebra. That is, if $A$ is endowed with a conjugate linear involution $J: A \to A$ such that $J(\mu(x, y)) = \mu(J(y), J(x))$, where $\mu$ is the multiplication on $A$.
\end{definition}

Let $Z$ be a two-dimensional Hermitian TQFT and let $h$ be the Hermitian form on $V = Z(S^1)$. Suppose $A = (V, \mu, \eta, \delta, \epsilon)$ is the complex commutative Frobenius algebra associated to $Z$ and suppose $\beta = \epsilon \mu$ is its pairing. We will also write $\beta$ for the nondegenerate symmetric bilinear form $\beta: V \times V \to \mathbb{C}$ given by $\beta(v, w) = \beta(v \otimes w)$.

\begin{theorem} \label{Frob_with_involution}
    A two-dimensional Hermitian TQFT $Z$ corresponds to a complex commutative $^*$-Frobenius algebra $A$ with involution $J$ given by $h(x, J(y)) = \beta(x, y)$ for all $x, y \in V$. Conversely, a $^*$-Frobenius algebra gives rise to a Hermitian TQFT in which the Hermitian structure is determined by the same formula. 
\end{theorem}

\begin{proof}
    Let $V$ denote the complex vector space $Z(S^1)$, which is also the underlying vector space of $A$. Let $B$ denote $\beta$ viewed as the isomorphism $B: V \to V^{*}$, $v \mapsto \beta(\cdot, v)$. Similarly, let $C$ denote $\gamma$ viewed as an isomorphism $V^{*} \to V$. Then by \Cref{herm_adj}, for all $v \in V$, $\alpha \in V^{*}$, 
    
    \begin{equation}\label{beta_gamma}
        h^{*}(B(v), \alpha) = h(v, C(\alpha)),
    \end{equation}
    where $h^{*}$ is the natural Hermitian form on $V^{*} \cong \overline{V}$ given by 
    \[
        h^{*}(h(\cdot, v), h(\cdot, w)) = \overline{h}(\overline{v}, \overline{w}) = h(w, v).
    \]
    Since gluing the cobordisms $\beta$ and $\gamma$ together on one of the boundary circles gives the cylinder, by the multiplicative axiom, $B \circ C = \operatorname{id}_{V^*}$. Thus, if $w = C(\alpha) \in V$, we can rewrite \cref{beta_gamma} as 
    \begin{equation}\label{beta_unitary}
        h^{*}(B(v), B(w)) = h(v, w).
    \end{equation}
    
    Let $H$ denote the conjugate linear isomorphism $V \to V^{*}$, $v \mapsto h(\cdot, v)$. Let $J = H^{-1} \circ B$ be the conjugate linear automorphism on $V$ such that $h(v, J(w)) = \beta(v, w)$ for all $v, w \in V$. Then \cref{beta_unitary} implies 
    \begin{align*}
        h(v, w) &= h^{*}(B(v), B(w))\\
        &= h^{*}(h(\cdot, J(v)), h(\cdot, J(w)))\\
        &= h(J(w), J(v))\\
        &= \beta(J(w), v)\\
        &= \beta(v, J(w))\\
        &= h(v, J^2(w)).
    \end{align*}
    The nondegeneracy of $h$ implies $J$ is an involution. 
    
    As a complex vector space,
    \[
        V \cong V_0 \otimes_{\mathbb{R}} \mathbb{C},
    \]
    where $V_0 = \{v \in V | J(v) = v\}$. The restriction of $h$ and $\beta$ to $V_0$ coincide, and for any $v, w \in V_0$, 
    \[
        h(v, w) = \beta(v, w) = \beta(w, v) = h(w, v),
    \]
    so this restriction is a real nondegenerate symmetric bilinear form.
    
    It remains to check the condition $J(\mu(x, y)) = \mu(J(y), J(x))$. By the Frobenius relation (\cref{eq_frob}), 
    \[
        \delta = \delta \mu (\eta \otimes \operatorname{id}) = (\operatorname{id} \otimes \mu)(\gamma \otimes \operatorname{id}).
    \]
    Also, by the multiplicative axiom, 
    \[
        \operatorname{id} = (\beta \otimes \operatorname{id})(\operatorname{id} \otimes \gamma).
    \]
    Suppose that $\gamma = \sum a_i \otimes b_i \in V \otimes V$. Applying the Hermitian axiom (\cref{herm}) to $\mu$ and $\delta$, for any $x, y, z \in V$,
    \begin{align*}
        h(\mu(x, y), z) 
        &= h_{V \otimes V}(x \otimes y, \delta(z))\\
        &= h_{V \otimes V}(x \otimes y, \sum a_i \otimes \mu(b_i, z))\\
        &= \sum h(x, a_i) h(y, \mu(b_i, z))\\
        &= \sum \overline{\beta(J(x), a_i)} h(y, \mu(b_i, z)) \\
        &= h(y, \mu(\sum \beta(J(x), a_i)b_i, z))\\
        &= h(y, \mu(J(x), z)). 
    \end{align*}
    This implies 
    \[
    \beta(J(\mu(x, y)), z) = \beta(J(y), \mu(J(x), z)) = \beta(\mu(J(y), J(x)), z).
    \] 
    By the nondegeneracy of $\beta$, $J(\mu(x, y)) = \mu(J(y), J(x))$. Therefore, $A$ is a commutative $^*$-Frobenius algebra with involution $J$. 
    
Conversely, suppose $A$ is a commutative $^*$-Frobenius algebra with involution $J$. Then $h(v, w) := \beta(v, J(w))$ gives a Hermitian structure on the underlying vector space $V$. Reversing the above computation shows that $h$ satisfies $h(\mu(x, y), z) = h_{V \otimes V}(x \otimes y, \delta(z))$. Also, $h(v, \eta) = \beta(v, \eta) = \epsilon(v)$. Thus, the Hermitian axiom holds for the two pairs of generators $(\mu, \delta)$ and $(\epsilon, \eta)$. Extending by the multiplicative axiom, the Hermitian axiom holds for all cobordisms. Therefore, the $^*$-Frobenius algebra structure supplements $Z$ to become a Hermitian TQFT.
\end{proof}
    
    The proof of \Cref{Frob_with_involution} suggests a simple structure for the commutative $^*$-Frobenius algebra associated to a two-dimensional Hermitian TQFT, which Sawin proves in (\cite{sawin1995}):
    
\begin{theorem}[Sawin (\cite{sawin1995})]\label{structure_thm}
    Let $A$ be the commutative $^*$-Frobenius algebra associated to a Hermitian TQFT $Z$, with involution $J$ as in \Cref{Frob_with_involution}. Then as a complex Frobenius algebra, 
    \[
        A \cong A_0 \otimes_{\mathbb{R}} \mathbb{C},
    \]
    where $A_0$ is the real commutative Frobenius algebra defined on the vector space $\{x \in V| J(x) = x\}$ with $\mu_{A_0}$ and $\epsilon_{A_0}$ taken to be the restrictions of $\mu_{A}$ and $\epsilon_{A}$ on $A_0$. In other words, $A$ is the complexification of $A_0$.
\end{theorem}

\begin{proof}
    Let $V$ be the underlying complex vector space of $A$ and $V_0$ be the real subspace fixed by $J$. From the proof of \cref{Frob_with_involution}, $V \cong V_0 \otimes_{\mathbb{R}} \mathbb{C}$. In addition, the restrictions of $\epsilon$ and $\mu$ to $V_0$ are real. Thus, $A_0 = (V_0, \epsilon|_{V_0}, \mu|_{V_0})$ is a real commutative Frobenius subalgebra of $A$. It follows that $A$ is the complexification of $A_0$. 
\end{proof}

\begin{remark}
    \Cref{structure_thm} shows that a Hermitian TQFT is essentially real, since $A_0$ is the (Hermitian) TQFT where $Z_0(S^1) = V_0$ and the cobordisms $Z_0(M)$ are taken to be the restrictions of $Z(M)$ to the ``real'' parts of the vector spaces. Equivalently, starting from a real TQFT $Z_0$, we can extend it to a Hermitian TQFT just by taking its complexification. 
\end{remark}

\section{The Algebraic Structure of Unitary TQFTs}\label{sec_unitary}
In the case when $Z$ is unitary (and two-dimensional), further structure exists on the corresponding $^*$-Frobenius algebra. Durhuus, Jonsson, and Sawin have studied unitary TQFTs algebraically in (\cite{durhuus1994,sawin1995}). In this section, we will prove natural extensions of the results in \Cref{sec_hermitian}, and provide a clarification and unification of the previous authors' results. 

\begin{definition}
    Suppose $A$ is a one dimensional real Frobenius algebra. By the nondegeneracy of the Frobenius form, $\epsilon(\eta) \neq 0$. We say $A$ is \textit{positive} if $\epsilon(\eta) > 0$ and \textit{negative} if $\epsilon(\eta) < 0$.
\end{definition}

The following theorem (\Cref{structure_thm_unitary}) is a natural extension of \Cref{structure_thm}. In (\cite{sawin1995}), Sawin presents it as a corollary to a result we present later (\Cref{c_star_Frob}).

\begin{theorem} \label{structure_thm_unitary}
    Let $A \cong A_{0} \otimes_{\mathbb{R}} \mathbb{C}$ be the Frobenius algebra associated to a unitary TQFT $Z$. If $\operatorname{dim}_{\mathbb{R}}{A_0} = n$, then $A_{0}$ is the product of $n$ one-dimensional positive Frobenius algebras. That is, $A_0 \cong \mathbb{R}^n$ where $\mathbb{R}^n$ is equipped with coordinate wise multiplication and $\epsilon(e_i) > 0$ for each idempotent $e_i$ spanning a copy of $\mathbb{R}$ in the product. 
\end{theorem}
\begin{proof}
    Since it suffices to deal with $A_0$, in this proof we will use $\beta$, $\mu$, etc., to denote the restriction of those on $A_0 \subset A$. From the proof of \Cref{Frob_with_involution}, $\beta = h|_{V_0}$. Since $Z$ is unitary, $\beta$ is a positive-definite symmetric bilinear form on $V_0$. Now consider the representation of $A_0$ through its action on $V_0$ by left multiplication:
    \[
        \varphi: A_0 \to \operatorname{End}(A_0), \quad \varphi(x)(y) = \mu(x, y).
    \]
    For all $x, y \in V_0$, $\varphi(x)\varphi(y) = \varphi(y)\varphi(x) = \mu(\mu(x, y), \cdot)$ by the commutativity and associativity of $\mu$. Furthermore, for all $v, w \in V_0$, 
    \[
        \beta(\varphi(x)(v), w) = \beta(v, \varphi(x)w).
    \]
    By the spectral theorem for symmetric operators, $\varphi(x)$ is diagonalizable. Since all $\varphi(x)$ commute, they are simultaneously diagonalizable. Thus, we can choose a basis $\{e_i\}$ of $V_0$ such that $\phi(A_0)$ lands in the subalgebra $\mathfrak{h} \cong \mathbb{R}^n$ of diagonal $n \times n$ matrices. Since $\varphi(x)(\eta) = \mu(x, \eta) = x$, $\varphi$ is faithful. Thus, by dimension counting, $A_0 \cong \mathbb{R}^n$.
    
    Finally, if we take the basis $\{e_i\}$ of $A_0$ with $\varphi(e_i) = \operatorname{diag}(0, \ldots, 1, \ldots, 0)$ the diagonal matrix with $1$ in the $i$-th entry, then $\mu(e_i, e_j) = \delta_{ij} e_i$. Thus, $\epsilon(e_i) = \beta(e_i, e_i) > 0$ (note that $\beta$ is not necessarily the standard dot product with respect to this basis). 
\end{proof}
\begin{remark}\label{complex_diag_matrices}
    In the proof above, the representation $\varphi$ extends to a faithful representation of $A$ as $n \times n$ complex diagonal matrices (as an associative algebra). The extra Frobenius form structure then assigns a positive ``weight'' to each of the $n$ entries. 
\end{remark}

\Cref{structure_thm_unitary} gives a classification of unitary TQFTs of dimension $n$: they are completely determined by the $n$ positive real numbers $\epsilon(e_1), \ldots, \epsilon(e_n)$ that define the Frobenius form. In (\cite{durhuus1994}), Durhuus and Jonsson also gives an equivalent classification using the spectrum of the handle operator $H = \mu \delta: V \to V$. 

\begin{corollary}[Durhuus and Jonsson (\cite{durhuus1994})]\label{Durhuus_thm}
    Let $A \cong \mathbb{R}^n \otimes_{\mathbb{R}} \mathbb{C}$ be the $^*$-Frobenius algebra associated to a unitary TQFT $Z$. Suppose $\{e_i\}$ is the standard basis of $\mathbb{R}^n$ with $\mu(e_i, e_j) = \delta_{ij} e_i$, $\epsilon(e_i) > 0$. Then each $e_i$ is an eigenvector of the handle operator $H = \mu \delta: V \to V$ with eigenvalue $\lambda_i = 1/\epsilon(e_i)$. 
\end{corollary}
\begin{proof}
    In (\cite{durhuus1994}), Durhuus and Jonsson start with Atiyah's axioms and compute the handle operator in terms of the generators. Then they examine its eigenvalues and argue that they are positive and they determine the TQFT. However, with \Cref{structure_thm_unitary}, we can take advantage of the nice structure of the Frobenius algebra associated to $Z$ and obtain the spectrum of $H$ more directly. 
    
    Again, it suffices to deal with the real algebra $A_0 = \mathbb{R}^n$. By the snake relation, 
    \begin{equation}\label{snake_relation}
        \operatorname{id} = (\operatorname{id} \otimes \beta)(\gamma \otimes \operatorname{id}).
    \end{equation}
    If $\gamma = \sum a_{ij} e_i \otimes e_j$, applying both sides of \cref{snake_relation} to a basis element $e_k$ gives
    \begin{align*}
        e_k &= (\operatorname{id} \otimes \beta)(\sum a_{ij} e_i \otimes e_j \otimes e_k)\\
        &= \sum a_{ij} \cdot e_i \cdot \beta(e_j \otimes e_k)\\
        &= \sum a_{ik} \epsilon(e_k) \cdot e_i.\\
    \end{align*}
    Thus, $\gamma = \sum_{i=1}^n \frac{1}{\epsilon(e_i)} e_i \otimes e_i = \sum_{i=1}^n \lambda_i e_i \otimes e_i$.
    Now we explicitly compute the handle operator,
    \begin{align*}
        H(x) 
        &= \mu (\delta (x))\\
        &= \mu (\operatorname{id} \otimes \mu) (\gamma \otimes \operatorname{id}) (x)\\
        &= \mu (\operatorname{id} \otimes \mu) \left(\sum_{i=1}^{n} \lambda_i e_i \otimes e_i \otimes x \right)\\
        &= \sum_{i=1}^{n} \lambda_i \mu(e_i, \mu(e_i, x))\\
        &= \mu \left(\sum_{i=1}^{n} \lambda_i \mu(e_i, e_i), x \right)\\
        &= \varphi \left(\sum_{i=1}^{n} \lambda_i e_i \right) x.
    \end{align*}
    Thus, the handle operator can be seen as left multiplication by the matrix $\operatorname{diag}(\lambda_1, \lambda_2, \ldots, \lambda_n)$, which has eigenvectors $e_i$ and eigenvalues $\lambda_i = 1/\epsilon(e_i)$. 
\end{proof}

In (\cite{sawin1995}), Sawin shows that the Frobenius algebra $A$ associated to a unitary TQFT $Z$ is a $C^*$-Frobenius algebra according to the following definition (\Cref{c_star_frob_def}). We prove this as another corollary to \Cref{structure_thm_unitary}. 

\begin{definition}\label{c_star_frob_def}
    As in (\cite{sawin1995}), a \textit{$C^*$-Frobenius algebra} is a $^*$-Frobenius algebra $A$ satisfying $\beta(J(x), x) > 0$ for all $x \neq 0$ (note that $J(\mu(J(x), x)) = \mu(J(x), x)$ by definition, so $\beta(J(x), x)$ is always real by \Cref{structure_thm}), together with a norm $|| \cdot ||$ such that for all $x \in V$, $||\mu(J(x),  x)|| = ||x||  \cdot ||J(x)||$.
\end{definition}

\begin{corollary}[Sawin (\cite{sawin1995})]\label{c_star_Frob}
    Let $A$ be the $^*$-Frobenius algebra associated to a unitary TQFT $Z$. Let $\varphi: A \to \operatorname{End}(V)$ be the faithful representation of $A$ as complex diagonal matrices as in \Cref{complex_diag_matrices}. Then $A$ admits a canonical $C^*$-Frobenius algebra structure with the norm given by 
    \[
        ||x|| = ||\varphi(x)||,
    \]
    for $x \in V$, where the norm on the right hand side is the operator norm on $n \times n$ complex matrices. 
\end{corollary}
\begin{proof}
    By \Cref{structure_thm_unitary}, as a $^*$-algebra, $A \cong \mathbb{R}^n \otimes_{\mathbb{R}} \mathbb{C}$, the algebra of $n \times n$ complex diagonal matrices, with involution $J$ given by complex conjugation of matrices. Under this representation, $A$ is endowed with the natural operator norm on matrices. (In particular, the norm is just the maximal entry in absolute value for diagonal matrices). It is easy to check that for all $x \in V$, 
    \begin{align*}
        ||\mu(J(x), x)|| 
        &= ||\overline{\varphi(x)} \varphi(x)||\\
        &= ||\varphi(x)|| \cdot ||\overline{\varphi(x)}||\\
        &= ||x|| \cdot ||J(x)||.
    \end{align*}
    
    Finally, for all $x \in V$, $x \neq 0$, 
    \[
        \beta(J(x), x) = \epsilon(\mu(J(x), x)) > 0
    \]
    since $\varphi(\mu(J(x), x))$ is a diagonal matrix whose entries are the squares of the norms of the entries of $\varphi(x)$ and by \Cref{structure_thm_unitary}, $\epsilon(e_i) > 0$ for all $e_i$ with $\varphi(e_i) = \operatorname{diag}(0, \ldots, 1, \ldots, 0)$. 
\end{proof}
\phantom{a}

\printbibliography
\end{document}